\documentclass[11pt]{article}
\usepackage{latexsym}
\usepackage[top=1in, bottom=1in, left=1in, right=1in]{geometry}

\usepackage{graphicx} 
\usepackage{color} 
\usepackage{amsmath, amsthm, amssymb}
\usepackage{enumerate}
\usepackage{float}
\usepackage{url}
\usepackage{stackrel}
\usepackage{mathrsfs,dsfont}
\usepackage{caption}
\usepackage{subcaption}
\usepackage{wrapfig}
\usepackage{bbm}
\usepackage{bm}

\newtheorem{theorem}{Theorem}[section]
\newtheorem{corollary}[theorem]{Corollary}
\newtheorem{proposition}[theorem]{Proposition}
\newtheorem{lemma}[theorem]{Lemma}

\newtheorem{definition}[theorem]{Definition}

\newtheorem{remark}[theorem]{Remark}
\newtheorem{problem}[theorem]{Problem}
\newtheorem{conjecture}[theorem]{Conjecture}

\usepackage[usenames,dvipsnames,svgnames,table]{xcolor}
\usepackage{url}
\newcommand{\St}{{S}}
\newcommand{\invtPoly}{\mathcal{P}}
\DeclareMathOperator{\im}{im}
\newcommand{\CRN}{chemical reaction network }

\newcommand{\R}{\mathbb{R}}
\newcommand{\Rnn}{\mathbb{R}_{\geq 0}}

\def\lra{\leftrightarrows}
\def\been{\begin{enumerate}}
\def\enen{\end{enumerate}}

\def\SS{\mathcal S}
\def\CC{\mathcal C}
\def\RR{\mathcal R}

\begin{document}

\title{A survey of methods for deciding whether a reaction network is multistationary}


\author{Badal Joshi  \and  Anne Shiu }

\maketitle




\abstract{Which reaction networks, when taken with mass-action kinetics, have the capacity for multiple steady states?  There is no complete answer to this question, but over the last 40 years various criteria have been developed that can answer this question in certain cases.  This work surveys these developments, with an emphasis on recent results that connect the capacity for multistationarity of one network to that of another. In this latter setting, we consider a network $N$ that is {\em embedded} in a larger network $G$, which means that $N$ is obtained from $G$ by removing some subsets of chemical species and reactions. 
This embedding relation is a significant generalization of the subnetwork relation.  
For arbitrary networks, it is not true that if $N$ is embedded in $G$, then the steady states of $N$ lift to $G$. Nonetheless, this does hold for certain classes of networks; one such class is that of 
fully open networks. This motivates the search for {\em embedding-minimal multistationary networks}: those networks which admit multiple steady states but no proper, embedded networks admit multiple steady states. We present results about such minimal networks, including several new constructions of infinite families of these networks.
}


%

\maketitle


\section{Introduction}
Reaction networks taken with mass-action kinetics arise in many scientific areas, from epidemiology (the SIR model) to population biology (Lotka--Volterra) to systems of chemical reactions.  Indeed, reaction networks often form a suitable 
modeling framework when the variables of interest take non-negative values. 
In the resulting mathematical model, a unique positive stable steady state in the mathematical model can underlie robustness in the corresponding biological system; 
conversely, the existence of multiple positive stable steady states can explain switching behavior in the biological system. 

This motivates the following important open question:  
which reaction networks, when taken with mass-action kinetics, have the capacity for multistationarity?  In other words, for which networks do there exist a choice of positive reaction rate constants and a choice of a forward-invariant set (or equivalently, a choice of non-negative initial concentrations) within which the corresponding dynamical system admits two or more (nondegenerate) steady states?  

At this time, there is no complete answer to this question; indeed, multistationary networks have not been completely catalogued.   Nonetheless, over the last 40 years various criteria have been developed that can either preclude or guarantee multistationarity for certain classes of networks.   The aim of the present article is to review these existing criteria, with an emphasis on recent results that connect the capacity for multistationarity of one network to that of certain related networks.
We will take a broad view, instead of focusing on specific families arising in chemistry or biochemistry.
For a historical survey of experimental  findings concerning multistationarity, see the book of Marin and Yablonsky~\cite[Chapter 8]{MY}.  

Among the criteria mentioned above for precluding multistationarity 
are
the injectivity criterion~\cite{ME1} and
 the deficiency zero and deficiency one theorems~\cite{FeinDefZeroOne}.
However, failing both the injectivity criterion and the conditions of the deficiency theorems is not sufficient for admitting multiple steady states. An instance of such a network is the following:
\begin{align*}
B \leftrightarrows 0  \leftrightarrows A \quad \quad 
3A + B  \to 2A + 2B 
\end{align*}
(For readers familiar with the literature, this network has deficiency one, and each of the two linkage classes have deficiency zero, so neither the deficiency zero nor deficiency one theorems applies.  Also, the injectivity criterion fails.) 
For this network, a simple calculation rules out multiple steady states, or the deficiency one algorithm can be applied to reach the same conclusion.

A somewhat more complicated example is the following:
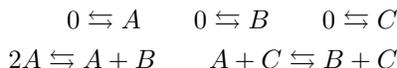
\begin{align*}
0  \leftrightarrows A \quad  \quad 0  \leftrightarrows B \quad  \quad 0  \leftrightarrows C \\
2A  \leftrightarrows  A+B \quad  \quad A+C \leftrightarrows B+C 
\end{align*}
In this case, the injectivity criterion fails, and the network has deficiency two -- which implies that the deficiency theorems do not apply, nor does the deficiency one algorithm. Nevertheless, the advanced deficiency algorithm rules out multistationarity in this case. 

By way of comparison, consider the next network:
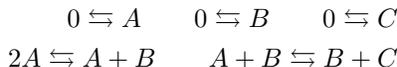
\begin{align*}
0  \leftrightarrows A \quad  \quad 0  \leftrightarrows B \quad  \quad 0  \leftrightarrows C \\
2A  \leftrightarrows  A+B \quad \quad A+B \leftrightarrows B+C
\end{align*}
where the only change from the previous network is that the complex $A+C$ in the last reaction has been replaced by the complex $A+B$. This time we find that the network does admit multiple steady states, because (1) it contains outflow reactions (such as $A \to 0$) for all three chemical species ($A, B, C$), and (2) a so-called ``CFSTR atom of multistationarity'' network, namely, $\{0  \leftrightarrows A~,~ 0  \leftrightarrows B ~,~ 0  \leftrightarrows C ~,~ B \to A+B ~,~ 2A \to A+B  \}$ is ``embedded'' in the network (see Theorem~\ref{thm:embedding-mss} and Corollary~\ref{cor:cfstr-1-rxn}).

Concerning more complicated networks, how can we assess multistationarity? In order to address this question, some recent work including our own has adopted the following strategy: instead of working with a specific network, can we relate the properties of two or more networks? Specifically, does a (large) network inherit some properties (such as multistationarity) of certain simpler related networks?

This question is intimately connected with and motivated by the problem of {\em model choice} (cf.\ \cite{FeliuWiuf}). A biological system displays a myriad of interactions among its components which take place at a wide variety of temporal scales. A good modeling design involves critical decisions about which components to incorporate within the model and which ones can be safely left out. In turn, any such decision about inclusion and exclusion of features is guided by the properties that the modeler is interested in examining as the model output, be it stability, multistationarity, oscillations, chaos, or some other behavior. 

Within a cell, a huge array of chemical species might interact with each other in a variety of ways. In this setting, suppose a model consists of a set of chemical reactions including one of the form $A + B \to 2C$; at least three major implicit assumptions typically underlie this choice of model:
\been
\item {\em The presence or absence of other chemical species, besides $A$, $B$ and $C$, does not affect the rate of the reaction in which $A$ and $B$ combine to produce two molecules of $C$.} However, future work may invalidate this assumption if further experimentation reveals that the reaction occurs only in the presence of an enzyme $E$, and thus the true chemical process is more accurately represented by one or more reactions involving $E$, such as $A+B + E \to 2C + E$.
\item {\em The chemical reactions that are included in the model are the only ones that are necessary to produce the dynamical range of behaviors under study, and inclusion of additional reactions will not significantly affect this dynamical repertoire.} Of course, it may turn out that species $A$ participates in an as-yet-unreported set of reactions involving a new substrate $F$, thus causing $A$ to be sequestered (and thus unavailable) for the original reaction $A+B \to 2C$ when the substrate $F$ is introduced into the reaction vessel or cell. 
\item {\em At the time scale of interest, the reaction $A + B \to 2C$ is a one-step reaction, in other words, there are no intermediate complexes that form as a molecule of $A$ combines with a molecule of $B$ to yield two molecules of $C$, and if any such intermediates do form, the time scale at which the intermediate species exist is much shorter than the time scale at which the species $A$, $B$, and $C$ are stable.} For instance, the reaction $A+B \to 2C$, when modeled at a finer temporal scale may be instead $A+B \lra AB \to 2C$. 
\enen
%
\noindent
Of course, assumptions such as these ones are unavoidable because biological processes tend to be quite complex, and even if the underlying mechanisms are known to a high level of detail, modeling every known feature is unrealistic. 
Indeed, underlying all modeling effort is an appeal to an implicit, redeeming principle that certain properties of the system are likely to be invariant to the level of detail included in the problem, once certain essential ingredients are incorporated into the model. 

In this article, we explore the validity of this implicit assumption by probing it closely in the case of multistationarity of reaction networks: how accurately are the system dynamics reflected by our choice of which reactions, species, and complexes are included in the model?  Can the number of steady states increase with the inclusion of more reactions or species? Can multistationarity be destroyed by increasing the level of detail included in the description of the model? Do the number of intermediate complexes (representing, for instance, various forms of an enzyme-substrate complex) in a model affect multistationarity?  We take up these questions and describe some results in this direction.

This article is organized as follows.
Section~\ref{sec:background} introduces chemical reaction networks and the dynamical systems they define. 
Section~\ref{sec:mss} reviews existing criteria pertaining to the capacity of a given network for multistationarity.
In Sections~\ref{sec:emb}--\ref{sec:atom}, we recall what it means for one network to be ``embedded'' in another, give some properties of this relation, and review existing criteria for when multiple steady states can be lifted from an embedded network. 
In Sections~\ref{sec:infinitely}--\ref{sec:infinitely2}, we demonstrate the existence of infinitely many embedding-minimal multistationary and multistable networks.  Finally, a discussion appears in Section~\ref{sec:openQ}.

\section{Background} \label{sec:background}
In this section we recall how a chemical reaction network gives rise to a dynamical system.  

We begin with an example of a {\em chemical reaction}: $A+B \to 3A + C$.  
%
In this reaction, one unit of chemical {\em species} $A$ and one of $B$ react 
to form three units of $A$ and one of $C$.  
The  {\em reactant} $A+B$ and the {\em product} $3A+C$ are called {\em complexes}. 
The concentrations of the three species, denoted by $x_{A},$ $x_{B}$, and
$x_{C}$, 
will change in time as the reaction occurs.  Under the assumption of {\em
mass-action kinetics}, species $A$ and $B$ react at a rate proportional to the
product of their concentrations, where the proportionality constant is the {\em reaction rate
constant} $\kappa$.  Noting that the reaction yields a net change of two units in
the amount of $A$, we obtain the first differential equation in the following
system, and the other two arise similarly:
\begin{align*}
\frac{d}{dt}x_{A}~&=~2\kappa x_{A}x_{B}~ \\
\frac{d}{dt}x_{B} ~&=~-\kappa x_{A}x_{B}~ \\
\frac{d}{dt} x_{C}~&=~\kappa x_{A}x_{B}~.
\end{align*}
 A {\em chemical reaction network}
consists of finitely many reactions.  
The mass-action differential equations that a network defines are comprised of a 
sum of the monomial contribution from the reactant of each 
chemical reaction in the network; these 
differential equations 
will be defined in equation~(\ref{eq:ODE-mass-action}).


\subsection{Chemical reaction systems}
We now provide precise definitions.  

\begin{definition}
A {\em chemical reaction network} $G=(\SS,\CC,\RR)$
consists of three finite sets:
\begin{enumerate}
\item a set of chemical {\em species} $\SS = \{A_1,A_2,\dots, A_s\}$, 
\item a set  $\CC = \{y_1, y_2, \dots, y_p\}$ of {\em complexes} (finite nonnegative-integer combinations of the species), and 
\item a set of {\em reactions}, which are ordered pairs of the complexes: $\RR \subseteq \CC \times \CC$.
\end{enumerate}
\end{definition}

Throughout this work, the integer unknowns~$p$, $s$, and $r$ denote the numbers of
complexes, species, and reactions, respectively.  
Writing the $i$-th complex as $y_{i1} A_1 + y_{i2} A_2 + \cdots + y_{is}A_s$ (where $y_{ij} \in \mathbb{Z}_{\geq 0}$ for $j=1,2,\dots,s$), 
we introduce the following monomial:
$$ x^{y_i} \,\,\, := \,\,\, x_1^{y_{i1}} x_2^{y_{i2}} \cdots  x_s^{y_{is}}~. $$
For example, the two complexes in the reaction $A+B \to 3A + C$ considered earlier give rise to 
the monomials $x_{A}x_{B}$ and $x^3_A x_C$, which determine the vectors 
$y_1=(1,1,0)$ and $y_2=(3,0,1)$.  
These vectors define the rows of a $p \times s$-matrix of nonnegative integers,
which we denote by $Y=(y_{ij})$.
Next, the unknowns $x_1,x_2,\ldots,x_s$ represent the
concentrations of the $s$ species in the network,
and we regard them as functions $x_i(t)$ of time $t$.

For a reaction $y_i \to y_j$ from the $i$-th complex to the $j$-th
complex, the {\em reaction vector}
 $y_j-y_i$ encodes the
net change in each species that results when the reaction takes
place.  
We associate to each reaction 
a positive parameter $\kappa_{ij}$, the rate constant of the
reaction.  In this article, we will treat the rate constants $\kappa_{ij}$ as positive
unknowns in order to analyze the entire family of dynamical systems
that arise from a given network as the $\kappa_{ij}$'s vary.  

A network can be viewed as a directed graph whose nodes are complexes and whose edges correspond to the reactions. A {\em linkage class} is a connected component of the directed graph: the complexes $y$ and $y'$ belong to the same linkage class if and only if there is a sequences of complexes $(y_0:=y, y_1, \ldots, y_{n-1}, y_n:=y')$ such that either $y_i \to y_{i+1}$ or $y_{i+1} \to y_{i}$ is a reaction for all $0 \le i \le n-1$. A network is said to be \emph{weakly reversible} if every connected component of the network is strongly connected.  A reaction $y_i \to y_j$ is {\em reversible} if its reverse reaction $y_j \to y_i$ is also in $\RR$; these reactions may be depicted as $y_i \rightleftharpoons y_j$.

The {\em stoichiometric matrix} 
$\Gamma$ is the $s \times r$ matrix whose $k$-th column 
is the reaction vector of the $k$-th reaction 
i.e., it is the vector $y_j - y_i$ if $k$ indexes the 
reaction $y_i \to y_j$.
The {\em reaction matrix} is the negative  of the transpose, $-\Gamma^t \in \mathbb{Z}^{r \times s}$, and the 
{\em reactant (source) matrix} is the $r \times s$ matrix whose $k$-th row is $y_i$ if $k$ indexes the 
reaction $y_i \to y_j$. 

The choice of kinetics is encoded by a locally Lipschitz function $R:\Rnn^s \to \R^r$ that encodes the reaction rates of the $r$ reactions as functions of the $s$ species concentrations.
The {\em reaction kinetics system} 
defined by a reaction network $G$ and reaction rate function $R$ is given by the following system of ODEs:
\begin{align} \label{eq:ODE}
\frac{dx}{dt} ~ = ~ \Gamma \cdot R(x)~.
\end{align}
For {\em mass-action kinetics}, which is the setting of this paper, the coordinates of $R$ are
$ R_k(x)=  \kappa_{ij} x^{y_i}$, 
 if $k$ indexes the reaction $y_i \to y_j$.  
A {\em chemical reaction system} refers to the 
dynamical system (\ref{eq:ODE}) arising from a specific chemical reaction
network $(\SS, \CC, \RR)$ and a choice of rate parameters $(\kappa^*_{ij}) \in
\mathbb{R}^{r}_{>0}$ (recall that $r$ denotes the number of
reactions) where the reaction rate function $R$ is that of mass-action
kinetics.  Specifically, the mass-action ODEs are the following:
\begin{align} \label{eq:ODE-mass-action}
\frac{dx}{dt} \quad = \quad \sum_{ y_i \to y_j~ {\rm is~in~} \RR} \kappa_{ij} x^{y_i}(y_j - y_i) \quad =: \quad f_{\kappa}(x)~,
\end{align}

The {\em stoichiometric subspace} is the vector subspace of
$\mathbb{R}^s$ spanned by the reaction vectors
$y_j-y_i$, and we will denote this
space by $\St$: 
\begin{equation} \label{eq:stoic_subs}
  \St~:=~ \mathbb{R} \{ y_j-y_i \mid  y_i \to y_j~ {\rm is~in~} \RR \}~.
\end{equation}
  Note that in the setting of (\ref{eq:ODE}), one has $\St = \im(\Gamma)$.
For the network consisting of the single reaction $A+B \to 3A + C$, we have $y_2-y_1 =(2,-1,1)$, which
means 
that with each occurrence of the reaction, two units of $A$ and one of $C$ are
produced, while one unit of $B$ is consumed.  This vector $(2,-1,1)$ spans the
stoichiometric subspace $\St$ for the network. 
Note that the  vector $\frac{d x}{dt}$ in  (\ref{eq:ODE}) lies in
$\St$ for all time $t$.   
In fact, a trajectory $x(t)$ beginning at a positive vector $x(0)=x^0 \in
\R^s_{>0}$ remains in the {\em stoichiometric compatibility class},
which we denote by
\begin{align}\label{eqn:invtPoly}
\invtPoly~:=~(x^0+\St) \cap \mathbb{R}^s_{\geq 0}~, 
\end{align}
for all positive time.  In other words, this set is forward-invariant with
respect to the dynamics~(\ref{eq:ODE}).    

A {\em steady state} of a reaction kinetics system~\eqref{eq:ODE} is a nonnegative concentration vector $x^* \in \Rnn^s$ at which the ODEs~\eqref{eq:ODE}  vanish: $f_{\kappa} (x^*) = 0$.  
A steady state $x^*$ is {\em nondegenerate} if ${\rm Im}\left( df_{\kappa} (x^*) \right) = \St$.  (Here, $df_{\kappa}(x^*)$ is the Jacobian matrix of $f_{\kappa}$ at $x^*$.)  A nondegenerate steady state is exponentially {\em stable} if each of the $\sigma:= \dim \St$ nonzero eigenvalues of $df_{\kappa}(x^*)$, viewed over the complex numbers, has negative real part. 
Also, we distinguish between {\em positive steady states} $x ^* \in \mathbb{R}^s_{> 0}$ and {\em boundary steady
states} $x^* \in  \left( \mathbb{R}^s_{\geq 0} \setminus \mathbb{R}^s_{> 0} \right)$.  
A system is {\em multistationary} (or {\em admits multiple steady states}) if there exists a stoichiometric
compatibility class $\invtPoly$ with two or more positive steady states. 
In the setting of mass-action kinetics, a network may admit multistationarity for all, some, or no choices of
positive rate constants $\kappa_{ij}$; if such rate constants exist, then we say that 
the network itself is {\em has the capacity for multistationarity} or, for short, is {\em multistationary}.  One focus of this work is on the following open question:
\begin{problem}
Which reaction networks are multistationary?
\end{problem}
This question is difficult in general: assessing whether a given network is multistationary means determining if the parametrized family of polynomial systems arising from the mass-action ODEs~\eqref{eq:ODE} ever admits two or more positive solutions.  Section~\ref{sec:mss} will describe results for certain families of networks, and Section~\ref{sec:openQ} provides an outlook on future progress on this question.

\subsection{The deficiency of a \CRN} \label{subsec:deficiency}
The  {\em deficiency} $\delta$ of a reaction network is an important invariant.   
For a reaction network, recall that $p$ denotes the number
of complexes. Also, let $l$ denote 
the number of linkage classes (connected components, as defined earlier). All networks considered in this article
have the property that each linkage class contains a unique {\em terminal strong linkage class}, i.e.~a maximal strongly connected subgraph in which there are no reactions from a complex in the subgraph to a complex outside the subgraph. 
In this case, Feinberg showed that the deficiency of the network can be computed in the following way: 
$$ \delta~:=~p-l-\dim(\St)~,$$
where $\St$ denotes the stoichiometric subspace~\eqref{eq:stoic_subs}.  
Note that in this case the deficiency depends only on the
reaction network and not on the specific values of the rate constants.
The deficiency of a reaction network is nonnegative
because it can be interpreted as the dimension of a
certain linear subspace \cite{Feinberg72}.

\subsection{Networks with flow reactions} \label{subsec:flow}

Here we give some definitions pertaining to flow reactions (which in some settings may be interpreted as production and degradation).
\begin{definition} \label{def:flow}
\begin{enumerate}
	\item 
	A {\em flow reaction} contains only one molecule; such a reaction is either an {\em inflow reaction} $0 \to X_i$ or an {\em outflow reaction} $X_i \to 0$.  A {\em non-flow reaction} is any reaction that is not a flow reaction.
	\item   
  A reaction network is a {\em continuous-flow stirred-tank reactor network} {\em (CFSTR network)} if it contains outflow reactions $X_i \to 0$ for all species $X_i$ of the network.
	\item 
	A reaction network is {\em fully open} if it contains inflow reactions  $0 \to X_i$ and outflow reactions $X_i \to 0$ for all species $X_i$ of the network.  (Thus, a fully open network is a CFSTR.)
  \end{enumerate}
\end{definition}

{\em Notation.}  The subnetwork of a network $G$ consisting of the non-flow reactions of $G$ will be called the {\em non-flow subnetwork} of $G$ and denoted by $G^{\circ}$.  Similarly, any reaction network $G=\{\SS, \CC, \RR \}$ is contained in a fully open network obtained by including all flow reactions; we call this the  {\em fully open extension} of $G$ and denote this CFSTR by $$\widetilde G := \{\SS,~\CC \cup \SS \cup \{0\},~\RR \cup \{X_i \leftrightarrow 0\}_{X_i \in \SS} \}~.$$

\section{Precluding or guaranteeing multiple steady states} \label{sec:mss}
Broadly speaking, results for assessing multistationarity of a given network arise from two areas: deficiency theory and injectivity theory.
This section surveys these results (see Table~\ref{table:review-methods}).  Additionally, many recent results relate the capacity for multistationarity of one network to that of another; these will be described in more depth in Section~\ref{subsec:embed-results}.  We note that, for simplicity, many of the results we describe are not stated here in their strongest or most general form (for instance, some results hold for more general kinetics than mass-action); instead, we refer the interested reader to the original papers.

\begin{table}[h]
\caption{Results for precluding or guaranteeing multistationarity}
\centering
\begin{tabular}{|l|l|l|l|l|}
\hline
                        & {\em Precluding} & {\em Guaranteeing} &   \\ 
                        & {\em multistationarity} & {\em multistationarity} & {\em Both}  \\ \hline
{\em Deficiency theory}       & Theorems~\ref{thm:def-0} and~\ref{thm:def-1}        &                                          & Advanced deficiency and                 \\ 
	~ & ~ & ~ & higher deficiency theories  \\
\hline
{\em Injectivity criterion}      & Theorems~\ref{thm:inj} and~\ref{thm:inj-cfstr}                 & Theorem~\ref{thm:det-opt}            &                           \\ \hline
{\em Using embedded networks} &                        				&           &     Theorem~\ref{thm:embedding-mss}                     \\ \hline
{\em Other approaches}       &  Monotone systems results                                 &                                          &   See Section~\ref{subsec:other-theories} \\ \hline
\end{tabular}
\label{table:review-methods}
\end{table}

\subsection{Deficiency theory} \label{subsec:def-theory}
Chemical reaction network theory, initiated by Feinberg, Horn, and Jackson beginning in the 1970s, aims to analyze reaction networks independently of the choice of rate constants.  One key area of progress is deficiency theory.  Here we review how these results can be used to preclude multistationarity. 
For details, we refer to the review paper of Feinberg~\cite[\S 6]{FeinDefZeroOne}.  The following two results are due to Feinberg.

\begin{theorem}[Deficiency zero theorem] \label{thm:def-0}
Deficiency-zero networks are not multistationary.  Moreover:
\begin{enumerate}
\item Every zero-deficiency network that is weakly reversible admits a unique positive steady state (for any choice of rate constants), and this steady state is locally asymptotically stable.
\item Every zero-deficiency network that is {\em not} weakly reversible admits {\em no} positive steady states (for any choice of rate constants).
\end{enumerate} 
\end{theorem} 

\begin{theorem}[Deficiency one theorem] \label{thm:def-1}
Consider a reaction network $G$ with linkage classes $G_1$, $G_2$, \dots, $G_l$.  Let $\delta$ denote the deficiency of $G$, and let $\delta_i$ denote the deficiency of $G_i$.  Assume that:
\begin{enumerate}
\item each linkage class $G_i$ has only one terminal strong linkage class,
\item $\delta_i \leq 1$ for all $i=1,2,\dots,l$, and
\item $\sum\limits_{i=1}^l \delta_i = \delta.$
\end{enumerate}
Then $G$ is not multistationary.
\end{theorem} 

Additionally, Ellison, Feinberg, and Ji developed advanced deficiency and higher deficiency theories (including the deficiency one and advanced deficiency algorithms) which in many cases can affirm that a given network admits multiple steady states or can rule out the possibility \cite{EllisonThesis,Fein95DefOne, JiThesis}.  All of these deficiency-related results have been implemented in {\tt CRN Toolbox}, freely available computer software developed by Feinberg, Ellison, Knight, and Ji~\cite{Toolbox}.  For readers who have a particular network of interest of small to moderate size, we recommend exploring what the Toolbox can say about your network.

\subsection{The injectivity criterion} \label{subsec:inj-theory}
In~\cite{ME1}, Craciun and Feinberg introduced a criterion that guarantees that a reaction network does not admit multiple seady states.  This test, which subsumed an earlier criterion of Schlosser and Feinberg~\cite{SchlosserFeinberg},
 is applicable even when the deficiency zero and deficiency one theorems are not, so this approach is complementary to the results from deficiency theory.  

The criterion arises from the following basic observation: letting $f_{\kappa}(x)$ denote the mass-action ODEs~\eqref{eq:ODE-mass-action} arising from a network and rate constants $\kappa = (\kappa_{ij})$, a sufficient condition to rule out multiple steady states is if the restricted map $f_{\kappa}|_{\invtPoly} : \invtPoly \to \St$ is injective for all $\kappa$ and all stoichiometric compatibility classes $\invtPoly$.  Such a network is said to be {\em injective} or to pass the {\em injectivity criterion} (also called the Jacobian criterion).  

The next result summarizes the contributions of many works that extended and gave equivalent formulations of the injectivity criterion.  We let $\sigma(x) \in \{+,-,0\}^n$ denote the sign vector of $x \in \mathbb{R}^n$.  Also, for an $m \times n$ matrix $M$ and nonempty subsets $I \subseteq \{1,2,\dots, m\}$ and $J \subseteq \{1,2,\dots, n\}$, we let $M_{I,J}$ denote the submatrix of $M$ formed by the rows indexed by $I$ and columns indexed by $J$. 
\begin{theorem}[Injectivity criterion~\cite{BP,ME1,ME_entrapped,ME2,ME3,JiThesis,ShinarFeinberg2012,WiufFeliu_powerlaw,signs}] \label{thm:inj} Let $G$ be a reaction network with stoichiometric matrix $\Gamma \in \mathbb{Z}^{s \times r}$, reactant matrix $M \in \mathbb{Z}_{\geq 0}^{r \times s}$, and stoichiometric subspace $\St = {\rm im}(\Gamma)$.  Let $f_{\kappa}(x):=\Gamma \cdot R(x)$ denote the mass-action ODEs~(\ref{eq:ODE-mass-action}) arising from $G$ and rate constants $\kappa=(\kappa_{ij}) \in \mathbb{R}_+^{r}$.  Then, the following are equivalent:
\begin{itemize} 
\item For all $\kappa \in \mathbb{R}^r_+$ and all stoichiometric compatibility classes $\invtPoly$, the map $f_{\kappa}|_{\invtPoly} : \invtPoly \to \St$ is injective.
\item For all $\kappa \in \mathbb{R}^r_+$ and all $x \in \mathbb{R}^s_+$, the Jacobian matrix of $f_{\kappa}(x)$ with respect to $x$
is injective on $\St$.
\item For all subsets $I \subseteq \{1,2,\dots, s\}$ and $J \subseteq \{1,2,\dots, r\}$ of size equal to the rank of $\Gamma$, the product ${\rm det}(\Gamma_{I,J})\cdot {\rm det}(M_{J,I})$ either is zero or has the same sign as all other nonzero such products, and moreover at least one such product is nonzero.
\item The sets $\ker(\Gamma)$ and $M(\sigma^{-1}(\sigma(\St)))$ have no nonzero sign vector in common.
\end{itemize}
If these equivalent conditions hold, then $G$ is not multistationary.
\end{theorem}

\begin{remark}
The original injectivity criterion of Craciun and Feinberg was for the case of CFSTRs (each chemical species has an associated outflow reaction)~\cite{ME1}.
Another proof of their result appeared in the context of geometric modeling~\cite{CGS}, 
and extended criteria were subsequently achieved for arbitrary networks~\cite{gnacadja_linalg,FeliuWiuf_MAK,JiThesis}.
For injectivity criteria for more general kinetics than mass-action, see the references listed in~\cite[Remark 3.5]{signs}.
\end{remark}

\begin{remark}
In some cases, the injectivity criterion can be translated to easy-to-check conditions on certain graphs arising from the network, namely, the species-reaction graph or the interaction graph~\cite{ME2,MinchevaCraciun2008,BanajiCraciun2010,HeltonDeterminant,ShinarFeinberg2013}.   We refer the interested reader to the review chapter of Craciun, Pantea, and Sontag~\cite{CPS}. 

Additionally, Ivanova obtained a result for precluding multistationarity that is based on the directed species-reaction graph \cite{ivanova}.  More precisely, the graph yields a system of inequalities involving the rate constants that, if consistent, guarantees that the mass-action system avoids multiple steady states.  For some networks, this inequality system is satisfied for all choices of rate constants, and thus multistationarity is precluded. For an example of such a network and an overview of these results, see~\cite[\S 5.4]{kinetic-book}.
\end{remark}

\begin{remark}
Determining whether a given network is injective (and thus precludes multistationarity) can be accomplished easily with the online software tool {\tt CoNtRol}~\cite{control} (which gives much more information as well).  We encourage the reader to give it a try.  
\end{remark}

In their original setting of CFSTRs, Craciun and Feinberg extended ideas underlying Theorem~\ref{thm:inj} to give a criterion that can guarantee multiple steady states; we will postpone stating this result (namely, Theorem~\ref{thm:det-opt}) until after introducing ``embedded networks'' in Section~\ref{sec:emb}.  A related result for guaranteeing multistationarity for non-injective networks appears in~\cite[\S 3.2]{signs}.


\subsection{Other approaches for assessing multistationarity} \label{subsec:other-theories}
Besides deficiency theory and injectivity theory, there are other approaches for determining whether a network is multistationary.  We mention several here, and then in Section~\ref{sec:emb} describe another approach: relating the capacity for multiple steady states of two or more similar networks.
First, monotone systems theory can be used to preclude multistationarity (for instance, see \cite{ADS10,banaji,BanajiM,BP,DB}).  In this context, if a dynamical system (such as a chemical reaction system) preserves some partial order, then existence and uniqueness of a positive steady state can be guaranteed.  For graph-theoretic criteria for monotonicity, see~\cite{CPS}.
%
Additionally, Conradi and Mincheva used degree theory to obtain a graph-theoretic condition for ruling out multiple steady states which is applicable for systems for which trajectories are bounded~\cite{ConradiMincheva14}.
Finally, Schlosser and Feinberg gave a condition guaranteeing multiple steady states for certain networks in which each linkage class consists of a pair of reversible reactions~\cite[Theorem 5.1]{SchlosserFeinberg}.
%



\section{Precluding/guaranteeing multistationarity by using embedded networks} \label{sec:emb}
In this section, we recall the definition of ``embedding'' and explain how this notion is useful for assessing multistationarity.
Embedded networks were introduced by the authors in~\cite{simplifying}, generalizing Craciun and Feinberg's notion of network {\em projections}~\cite[\S 8]{ME3}.  The embedding relation generalizes the subnetwork relation --  a subnetwork $N$ is obtained from a reaction network $G$ by removing a subset of reactions (alternatively, setting some of the reaction rates to be 0), while an embedded network is obtained by removing a subset of reactions or a subset of species (alternatively, setting the stoichiometric coefficients of those species to be 0) or both. 
For instance, removing the species $B$ from the reaction $A+B \to A+C$ results in the reaction $A \to A+C$. 

In some cases, removing species results in a {\em trivial reaction} -- where the source and product complex are identical. For instance, removal of  both $B$ and $C$ from $A+B \to A+C$ results in the trivial reaction $A \to A$. So, after removing species, any trivial reactions and any copies of duplicate reactions are discarded. 
\begin{definition} Given a set of reactions $\RR$ and a set of species $\SS$, we define the {\em restriction of $\RR$ to $\SS$}, denoted $\RR |_{\SS}$, to be the set of reactions obtained from $\RR$ after implementing the following steps: (1) set the stoichiometric coefficients of species that are 
not in $\SS$ equal to 0, and then (2) discard any {\em trivial reactions}, that is, reactions of the form $C_i \to C_i$, where the source and product complexes are identical.  
\end{definition}

As for removing both species and reactions, consider a network $G = (\SS,\CC,\RR)$ and 
a subset of reactions $\{y \to y'\} \subseteq \RR$ and a subset of species $\{X_i\} \subseteq \SS$ that are to be removed. Then the reactions in the resulting embedded network  $N$ 
will be denoted by $\RR_N = \left(\RR \setminus  \{ y \to y' \}\right)|_{\SS \setminus \{X_i\}}$ in Definition~\ref{def:emb} -- operationally, we obtain $N$ by first removing all the reactions $\{ y \to y' \}$ from $G$, and then removing the species $\{X_i\}$ from $G$, and finally restricting the set of reactions to the remaining species. 

\begin{definition} \label{def:emb}
Let $G = (\SS,\CC,\RR)$. The {\em embedded network} $N$ of $G$ obtained by removing the set of reactions $\{ y \to y' \} \subseteq \RR$ and the set of species  $\{X_i\} \subseteq \SS$ is 
\[
N = \left(\SS|_{\CC|_{\RR_N}}, \CC|_{\RR_N}, \ \RR_N := \left(\RR \setminus  \{ y \to y' \}\right)|_{\SS \setminus \{X_i\}}\right)~.
\]
\end{definition}
Note that the subset of reactions that remain in the embedded network $N$ is not necessarily a complement of $\{ y \to y' \}|_{\SS \setminus \{X_i\}}$, because some reactions may be removed because they are either trivial reactions or duplicate reactions. Similarly, the subset of species that are in the embedded network need not be a complement of the subset of the removed species $\{X_i\}$. 
The following result establishes some basic properties about embedded networks; the proof appears in the Appendix.

\begin{proposition} \label{prop:emb-properties}
Consider networks $N$ and $G$, with stoichiometric subspaces $S_N$ and $S_G$ and deficiencies $\delta_N$ and $\delta_G$, respectively. 
\been
\item If $N$ is an embedded network of $G$, then $\dim(S_N) \le \dim(S_G)$. 
\item If $N$ is a subnetwork of $G$, and each linkage class of $N$ and each linkage class of $G$ contains a unique terminal strong linkage class, then $\delta_N \le \delta_G$. 
\enen
\end{proposition}

\subsection{Existing results for assessing multistationarity of one network from another} \label{subsec:embed-results}
In Section~\ref{sec:mss}, we reviewed results for directly assessing multistationarity of a given network.  In this section, we recall what is known about how the capacity for multistationarity of a given network $G$ is related to that of one of its embedded networks $N$.  Most of these results ``lift'' steady states of $N$ to $G$, so in these settings\footnote{In general, it may be impossible to lift multiple steady states from an embedded network $N$ to $G$~\cite[Example 4.5]{joshi2012atoms}.}, if $N$ is multistationary then $G$ is too, or, equivalently, if $G$ is not multistationary then neither is $N$.  Therefore these results can be used to preclude or to guarantee multistationarity. We summarize what is known in this setting in Theorem~\ref{thm:embedding-mss}, of which parts 1 and 3 are due to the authors of the present work~\cite{joshi2012atoms}, part 2 is due to Craciun and Feinberg~\cite[Theorem~2]{ME_entrapped}, 
and part 4 is due to Feliu and Wiuf~\cite{FeliuWiuf}.
\begin{definition}
The {\em induced} network obtained by removing complexes $\{ \CC_1, \dots, \CC_k\}$ is obtained by removing those complexes and then replacing any pairs of reactions of the form $\CC' \to \CC_i \to \CC''$ with a reaction $\CC' \to \CC''$ (and removing duplicate reactions, as necessary).   Also, an {\em intermediate} complex has the form $X_i$, where $X_i$ is a species with total molecularity 1 (so it appears with stoichiometry 1 in the intermediate complex, and appears in no other complexes).
\end{definition}
\begin{theorem} \label{thm:embedding-mss}
Let $N$ and $G$ be reaction networks that are related in at least one of the following ways:
	\begin{enumerate}
	\item $N$ is a subnetwork of $G$ which has the same stoichiometric subspace as $G$,
	\item $G$ is the fully open extension of $N$, i.e.\ $G=\widetilde N$ (in this case, $N$ is a subnetwork of $G$),
	\item $N$ is a CFSTR embedded in a fully open network $G$,
	\item $N$ is an induced network of $G$ obtained by removing one or more intermediates.
	\end{enumerate}
Then, if $N$ admits $m$  positive nondegenerate steady states (for some choice of rate constants), then $G$ admits at least $m$ positive nondegenerate steady states (for some choice of rate constants).  Also, if $N$ admits $q$ positive, stable steady states, then $G$ admits at least $q$ positive, stable steady states.
\end{theorem} 

\begin{remark}
In the context of Theorem~\ref{thm:embedding-mss}, $G$ might admit more steady states than $N$.  For instance, a single reversible reaction has a unique positive steady state (by part 1 of the deficiency zero theorem), but making it irreversible yields no positive steady states (by part 2 of the theorem).
\end{remark}

\begin{remark} \label{rmk:partial-order}
Each of the cases of Theorem~\ref{thm:embedding-mss} defines a partial order on the set of reaction networks.  For case~4, Feliu and Wiuf called a minimal network with respect to this partial order (i.e.\ a network with no intermediates) a ``core model'', and then ``extension models'' are all networks obtained from the core model by adding intermediates.  For case 3 (CFSTRs), the authors gave the name ``CFSTR atoms of multistationarity'' to the minimal networks \cite{joshi2012atoms} (see Definition~\ref{def:atom} below), and now Theorem~\ref{thm:embedding-mss} motivates us to consider the general partial order formed by the embedding relation.
\end{remark}

\begin{definition} \label{def:atom}
\been
\item A fully open network is a {\em CFSTR atom of multistationarity} if 
it is minimal with respect to the embedded network relation among all fully open networks that admit multiple nondegenerate positive steady states.
\item An {\em embedding-minimal multistationary network} is minimal with respect to the embedded network relation among all networks that admit multiple nondegenerate positive steady states.
\item A fully open network is a {\em CFSTR atom of multistability} if 
it is minimal with respect to the embedded network relation among all fully open networks that admit multiple nondegenerate positive stable steady states.
\item An {\em embedding-minimal multistable network} is minimal with respect to the embedded network relation among all networks that admit multiple nondegenerate positive stable steady states.
\enen
\end{definition}
Note that any embedding-minimal multistationary network that is fully open is a CFSTR atom.  

In Section~\ref{sec:atom}, we will begin to answer the following open problems in the case of certain small networks: 
\been
\item Catalogue all CFSTR atoms of multistationarity.
\item Catalogue all embedding-minimal multistationary networks.
\enen

\begin{remark}
In order to implement Theorem~\ref{thm:embedding-mss}, the authors propose the following ``wish list'':
\been
\item A database of known CFSTR atoms of multistationarity, together with:
\item Software for automatically checking whether a given CFSTR contains as an embedded network a member of the database (which would  imply that the larger CFSTR admits multiple steady states).
\enen
\end{remark}

Going beyond item 1 of Theorem~\ref{thm:embedding-mss}, we now 
 consider the case that the stoichiometric subspaces of $G$ and its subnetwork $N$ do {\em not} coincide\footnote{In this case, 
	when the two stoichiometric subspaces do not coincide, there is one result of the 
	form in Theorem~\ref{thm:embedding-mss}: Conradi {\em et al.} proved that under certain conditions  
	one can lift multiple steady states from certain subnetworks called ``elementary 
	flux modes''~\cite[Supporting Information]{Conradi_subnetwork}.}. Then the next (easy) result states that a sufficient condition for being able to lift positive steady states (or for the larger network to have any positive steady states at all) is that some positive linear combination of the new reaction vectors must be in the stoichiometric subspace of $N$.  In particular, adding a single reaction to $N$ that is {\em not} in the stoichiometric subspace of $N$ results in a network with no positive steady states.
\begin{proposition} \label{prop:nec-cond-subnetwork-lift}
Let $N$ be a subnetwork of $G$ with subspaces $S_N$ and $S_G$, respectively.  Let $T=\{y_i \to y_i' \}_{i=1, \dots, q}$ denote the reactions in $G$ but not in $N$.  
If $S_N \neq S_G$ and no positive linear combination of the 
reaction vectors of $T$ is in $S_N$, i.e.\
$$\sum_{i=1}^q \alpha_i \left( y_i' -y_i \right) \notin S_N$$
for all $(\alpha_i) \in \mathbb{R}_{>0}^{q}$, 
then $G$ is not multistationary (and in fact $G$ has no positive steady states).
\end{proposition}

\begin{proof}
Let $\{y_i \to y_i'\}_{i=q+1, \dots, r}$ denote the reactions in $N$.  
If $x^*$ is a positive steady state of the mass-action system arising from $G$ and rate constants $\kappa$, then letting $\alpha_i := \kappa_{y_i \to y_i'} (x^*)^{y_i}$ and rearranging the steady state equations yields $\sum_{i=1}^q \alpha_i \left( y_i' -y_i \right) = \sum_{j=q+1}^r \alpha_i \left( y_i' -y_i \right)  \notin S_N$.
\end{proof}

\begin{remark}
Underlying Proposition~\ref{prop:nec-cond-subnetwork-lift} and its proof is the following necessary condition for a network to admit positive steady states: the reaction vectors must be positively dependent~\cite[Remark 2.1]{Fein95DefOne}.
\end{remark}


\begin{remark}
In Theorem~\ref{thm:embedding-mss}, the number of steady states of a ``smaller'' network $N$ is a lower bound on the number for a ``larger'' network $G$.  The only result known to us where (roughly speaking) a ``smaller'' network gives an {\em upper bound} on the number of steady states of a ``larger'' network is due to Feliu and Wiuf~\cite{FeliuWiuf}.  They considered the (larger) ``extension models'' of some ``core model'' network (recall Remark~\ref{rmk:partial-order}); then a ``canonical model'' is obtained by adding certain the reactions to the core model.  If this canonical model is not multistationary, then every extension model of the core model also is not multistationary~\cite[Corollary~6.1]{FeliuWiuf}. 
\end{remark}

\subsection{Using square embedded networks to apply the injectivity criterion to CFSTRs}
In Theorem~\ref{thm:embedding-mss}, we saw that embedded networks are useful for both ``ruling out'' and ``ruling in'' multistationarity.
In this section, we will see that certain embedded networks (namely, square embedded networks) can be used to elaborate on the injectivity criterion (Theorem~\ref{thm:inj}) for the case of a CFSTR, and again give results both for precluding and guaranteeing multiple steady states (Theorems~\ref{thm:inj-cfstr} and~\ref{thm:det-opt}).

\begin{definition} \label{def:square-or}
\been
\item A network is {\em square} if it has the same number of reactions and species.
\item Consider a square reaction network $G = (\SS, \CC, \RR)$ with $\RR= \{y_1 \to y_1', \ldots, y_s \to y_s' \}$ .
The {\em orientation of $G$} is ${\rm Or}(G) =\det ([y_1, \ldots, y_s]) \det ([y_1 - y_1', \ldots, y_s - y_s'])$.
\enen
\end{definition}

It is straightforward now (by translating part 3 of Theorem~\ref{thm:inj}) to give the following equivalent formulation of the injectivity criterion: all square embedded networks of size equal to the dimension of the stoichiometric subspace and with nonzero orientation have the same orientation (and at least one is nonzero). 
We will state this in Theorem~\ref{thm:inj-cfstr} for the case of CFSTRs.  
%

Any square embedded network that contains an inflow reaction has zero orientation. On the other hand, removing a species which participates in an outflow reaction does not change the orientation of the resulting SEN.
These two observations significantly simplify the calculation of orientation of all $s$-square networks of a CFSTR in $s$ species. 
In fact, we only need to examine the square embedded networks of $G^{\circ}$ where $G^{\circ}$ is the non-flow subnetwork of $G$. This is because there is an orientation-preserving one-to-one map between the subnetworks of $G$ containing $s$ reactions at least one of which is a non-flow reaction, and the square embedded networks of $G^\circ$.

We need the following definitions in order to state this result precisely. 

\begin{definition} \label{def:tm}
Consider a reaction network whose set of non-flow reactions is given by:
	\begin{align*}
	y_1 & \to y_1' \quad & y_2 & \to y_2' \quad & \ldots & \quad & y_l & \to y_l'  \\
	y_{l+1} & \leftrightarrow y_{l+1}' \quad & y_{l+2} & \leftrightarrow y_{l+2}' \quad & \ldots & \quad & y_{l+k} & \leftrightarrow y_{l+k}' ~,
	\end{align*}
where none of the first $l$ reactions is reversible.  The {\em total molecularity} of a species $X_i$ in the network
is the following non-negative integer:
$	\operatorname{TM}(X_i)  =  \sum_{j=1}^{l+k} \left( y_{ji} + y_{ji}' \right), $
	where $y_{ji}$ is the stoichiometric coefficient of species $X_i$ in the complex $y_j$.
\end{definition}

\begin{definition}\label{def:relevant}
A network $N$ is {\em relevant} if 	
	it satisfies the following properties:
	\begin{itemize}
		\item $N$ has no outflows, inflows, generalized inflow reactions $0 \rightarrow \sum_i a_i X_i$, or generalized outflows $a X_i \to b X_i$ (where $0 \leq b \leq a$). 
		\item $N$ does not contain a pair of reversible reactions.
		\item Each species appears in at least two complexes and in at least one reactant complex. 
		\item At least one species of $N$ has a total molecularity of at least 3.
	\end{itemize}
\end{definition}
Non-relevant square embedded networks  
may be ignored for the purposes of establishing injectivity. This observation contributes to the following result: 
\begin{theorem}[\cite{ME1}, Theorems 4.9 and 5.1 of \cite{simplifying}] \label{thm:inj-cfstr}
For a CFSTR $G$ with $s$ species, with $G^{\circ}$ its subnetwork of non-flow reactions, the following are equivalent:
\begin{enumerate}
	\item $G$ passes the injectivity criterion (i.e.\ satisfies the equivalent conditions of Theorem~\ref{thm:inj}).
	\item Each $s$-square embedded network of $G$ has non-negative orientation. 
	\item Each square embedded network of $G^{\circ}$ has non-negative orientation.
	\item Each {\em relevant} square embedded network of $G^{\circ}$ has non-negative orientation.
\end{enumerate}
If these equivalent conditions hold, then $G$ is not multistationary.  
\end{theorem}

The following is a partial generalization of Theorem~\ref{thm:inj-cfstr}: for any network $G$ (not necessarily a CFSTR) with stoichiometric subspace $\St$, the network $G$ is injective if and only if each $(\dim \St)$-square embedded network of $G$ has non-negative orientation.  This result is simply a translation of item~3 of Theorem~\ref{thm:inj}.

Finally, we present a criterion that can certify multistationarity in a CFSTR: if it fails the injectivity criterion (that is, the equivalent conditions in Theorem~\ref{thm:inj-cfstr} do not hold), then sometimes a negatively oriented SEN can be used to guarantee multiple steady states.  The result is due to Craciun and Feinberg.
\begin{theorem}[Determinant optimization method, Theorem 4.2 of~\cite{ME1}] \label{thm:det-opt}
For a CFSTR $G$ with $s$ species, with $G^{\circ}$ its non-flow subnetwork, assume that there exists a negatively oriented square embedded network of $G^{\circ}$ of size $s$, consisting of reactions $y_i \to y_i'$ (for $i=1,2,\dots , s$), and that some positive linear combination of the negative reaction vectors $y_i - y_i'$ is a positive vector (i.e.\ there exist $\eta_1, \eta_2, \dots, \eta_s>0$ such that $\sum\limits_{i=1}^s \eta_i (y_i - y_i') \in \mathbb{R}^s_{>0}$).  Then, the fully open extension of $G$ is multistationary.
\end{theorem}

Building on Theorem~\ref{thm:det-opt}, Feliu recently gave a new method based on injectivity for either guaranteeing or precluding multistationarity for general (not necessarily CFSTR) networks~\cite[\S 2.4]{feliu}.

\section{CFSTR atoms of multistationarity} \label{sec:atom}
For certain small reaction networks, a complete characterization of multistationarity is known.  Here we report on these results, including 
what is known about small CFSTR atoms of multistationary (recall Definition~\ref{def:atom}).

\subsection{Classification of multistationary, fully open networks with 1 non-flow reaction}
Fully open reaction networks with a single non-flow reaction (irreversible or reversible) form an infinite family of networks. Strikingly, however, the multistationary networks in this family have been completely determined.  Moreover, their characterization depends entirely on a simple arithmetic relation on the stoichiometric coefficients, which is ascertained at a glance.  The following result is due to the first author~\cite{joshi2013complete}, and its proof uses primarily the deficiency one algorithm, along with deficiency theory (Theorems~\ref{thm:def-0}--\ref{thm:def-1}).

\begin{theorem}[Classification of fully open networks with one non-flow reaction \cite{joshi2013complete}] \label{thm:onerxn}
Let $n$ be a positive integer.  Let $a_1, a_2, \dots, a_n, b_1, b_2, \dots, b_n$ be nonnegative integers.
\begin{enumerate}
\item The (general) fully open network with one irreversible non-flow  reaction and $n$ species: 
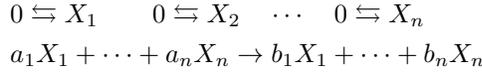
\begin{align*}
&0 \leftrightarrows X_1 \quad \quad 
 0 \leftrightarrows X_2 \quad  \cdots \quad
0 \leftrightarrows X_n \quad \quad  \\
&a_1X_1 + \cdots + a_n X_n \to b_1X_1 + \cdots + b_n X_n 
\end{align*}
is multistationary if and only if $\sum_{i: b_i > a_i} a_i >1$. 
\item The (general) fully open network with one reversible non-flow  reaction and $n$ species:
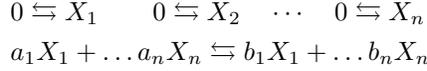
\begin{align*}
&0 \leftrightarrows X_1 \quad \quad 
 0 \leftrightarrows X_2 \quad  \cdots \quad
0 \leftrightarrows X_n \quad \quad  \\
&a_1X_1 + \ldots a_n X_n \leftrightarrows b_1X_1 + \ldots b_n X_n 
\end{align*}
is multistationary if and only if $\sum_{i: b_i > a_i} a_i >1$ or $\sum_{i: a_i > b_i} b_i >1$. 
\end{enumerate}
\end{theorem}   

Theorem~\ref{thm:onerxn} can be used to determine all CFSTR atoms with exactly one non-flow reaction:
\begin{corollary}[Classification of CFSTR atoms of multistationarity with one non-flow reaction] \label{cor:cfstr-1-rxn}
Up to symmetry, the CFSTR atoms of multistationarity that have only one non-flow reaction are the following:
\been
	\item  $\{0 \leftrightarrow  A, ~ mA \to nA \}$, where $m$ and $n$ are positive integers satisfying $ n > m > 1$. 
	\item $\{0 \leftrightarrow  A, ~0 \leftrightarrow  B, ~A + B \to mA + nB \}$,  where $m$ and $n$ are positive integers with  $m>1$ and $n  > 1$. 
\enen
\end{corollary}
In Section~\ref{subsec:2-fam}, we will see that all the CFSTR atoms in Corollary~\ref{cor:cfstr-1-rxn} are in fact embedding-minimal multistationary networks (Theorem~\ref{thm:infin}).

\subsection{Classification of multistationary, fully open networks with 2 non-flow reactions}
A {\em monomolecular} complex has the form $X_i$ for some species $X_i$, while a bimolecular complex has the form $2X_i$ or $X_i + X_j$.  A complex is {\em at most bimolecular} if it is the zero complex, monomolecular, or bimolecular.  A network is {\em at most bimolecular} if every complex in the network is at most bimolecular.

By Corollary~\ref{cor:cfstr-1-rxn}, there are no at-most-bimolecular CFSTR atoms with only one non-flow reaction.  An enumeration of at-most-bimolecular CFSTR atoms with two non-flow reactions (irreversible or reversible) was completed by the authors in~\cite{joshi2012atoms}:
\begin{theorem}[Classification of CFSTR atoms of multistationarity with two non-flow reactions and at-most-bimolecular complexes~\cite{joshi2012atoms}] \label{thm:2-rxn}
Up to symmetry, there are 11 CFSTR atoms of multistationarity that have only two non-flow reactions (irreversible or reversible) and complexes that are at most bimolecular:
\been
\item $\{  0 \leftrightarrows A, ~ 0 \leftrightarrows B, ~
A \to 2A,~
A+B \to 0
  \}$
\item $\{  0 \leftrightarrows A, ~ 0 \leftrightarrows B, ~
A \to 2A,~
A \leftrightarrows  2B
  \}$
\item $\{  0 \leftrightarrows A, ~ 0 \leftrightarrows B, ~ 0 \leftrightarrows C, ~
A \to 2A,~
A \leftrightarrows  B+C
  \}$
\item $\{  0 \leftrightarrows A, ~ 0 \leftrightarrows B, ~
A \to A+B,~
2B \to A
  \}$
\item $\{  0 \leftrightarrows A, ~ 0 \leftrightarrows B, ~
A \to A+B,~
2B \to 2A 
  \}$
\item $\{  0 \leftrightarrows A, ~ 0 \leftrightarrows B, ~
A \to A+B \to 2A
  \}$
\item $\{  0 \leftrightarrows A, ~ 0 \leftrightarrows B, ~
A \to A+B,~
2B \to A+B
\}$
\item $\{  0 \leftrightarrows A, ~ 0 \leftrightarrows B, ~ 
B \to 2A \to A+B
  \}$
\item $\{  0 \leftrightarrows A, ~ 0 \leftrightarrows B, ~
B \to 2A \to 2B
  \}$
\item $\{  0 \leftrightarrows A, ~ 0 \leftrightarrows B, ~ 0 \leftrightarrows C, ~
A \to B+C \to 2A
  \}$
\item $\{  0 \leftrightarrows A, ~ 0 \leftrightarrows B, ~
A+B \to 2A,~
A \to 2B
\}$
\enen
Thus, a CFSTR with two non-flow reactions (irreversible or reversible) admits multiple nondegenerate positive steady states if and only if it contains one of the eleven networks above as an embedded network.
\end{theorem}

\begin{remark} The first network among the listed two-reaction CFSTR atoms $\{  0 \leftrightarrows A, ~ 0 \leftrightarrows B, ~
A \to 2A,~
A+B \to 0
  \}$, along with the one-reaction atom $\{0 \leftrightarrow  A, ~ 2A \to 3A \}$ have received close examination in the condensed matter theory literature because they exemplify important principles regarding multistationarity and systems that undergo phase transitions \cite{grassberger1982phase,schlogl1972chemical}. 
\end{remark}

\section{Existence of infinitely many embedding-minimal multistationary networks} \label{sec:infinitely}
Corollary~\ref{cor:cfstr-1-rxn} demonstrated that there are infinitely many 
CFSTR atoms of multistationarity with one non-flow reaction.  In Section~\ref{subsec:2-fam}, we demonstrate that all the CFSTR atoms listed in that result are in fact embedding-minimal multistationary networks (Theorem~\ref{thm:infin}), showing that there are infinitely many such networks.
The molecularities of the complexes in those networks are arbitrarily large, which is unrealistic, so Section~\ref{subsec:seq} introduces a family of ``sequestration networks'' which are conjectured in Section~\ref{subsec:seq-infin} to form an infinite family of networks with at-most-bimolecular complexes that are both CFSTR atoms and embedding-minimal multistationary networks.

\subsection{New result: two infinite families with one non-flow reaction} \label{subsec:2-fam}
Here we see that even among networks with only one species, there exist embedding-minimal multistationary networks.  In fact, every CFSTR atom with only one non-flow reaction (enumerated earlier in Corollary~\ref{cor:cfstr-1-rxn}) is an embedding-minimal multistationary network:

\begin{theorem} \label{thm:infin}
For positive integers $m$ and $n$, consider the following networks: $G_{m,n} := \{0 \leftrightarrow  A,~ mA \to nA \}$ and 
$H_{m,n} := \{0 \leftrightarrow  A, ~0 \leftrightarrow  B,~ A + B \to mA + nB \}$. Then 
	\been
	\item $G_{m,n}$ is a CFSTR atom of multistationarity and an embedding-minimal multistationary network if and only if $ n > m > 1$. 
	\item $H_{m,n}$ is a CFSTR atom of multistationarity and an embedding-minimal multistationary network if and only if $ m>1$ and $n  > 1$.
	\enen
\end{theorem} 
\begin{proof} That $G_{m,n}$ is multistationary if and only if $ n > m > 1$ follows from part 1 of Theorem~\ref{thm:onerxn}.  Additionally, those networks are CFSTR atoms of multistationarity (Corollary~\ref{cor:cfstr-1-rxn}). Thus, it remains only to show that no proper embedded network of  $G_{m,n}$ (when $n>m>1$) is multistationary.  We begin by letting  $s$, $l$, and $k$ denote the rates of the reactions $0 \to A$, $A \to 0$, and $mA \to nA$, respectively. 
Then, the single mass-action ODE~\eqref{eq:ODE-mass-action} arising from $G_{m,n}$ is $\frac{da}{dt} = s - la + (n-m)ka^m$, where $a$ denotes the concentration of the species $A$; thus the number of positive steady states is the number of positive real roots of the univariate polynomial 
\begin{align} \label{eq:poly-for-G}
s - la + (n-m)ka^m~.
\end{align}
The only nontrivial embedded networks of $G_{m,n}$ arise from removing reactions (as $G_{m,n}$ has only one species), i.e.\ setting one more of the reaction rates $s$, $l$, or $k$ to zero.  It is straightforward to see that this results in the polynomial~\eqref{eq:poly-for-G} having at most one positive real root.  

By analogous reasoning for $H_{m,n}$, we need only show that no proper embedded network of  $H_{m,n}$ (when $m,n>1$) is multistationary.  We begin by denoting the rates of reactions $0 \to A$, $A \to 0$, $0 \to B$, $B \to 0$, and $A+B \to mA + nB$ by $s_A$, $l_A$, $s_B$, $l_B$, and $k$, respectively. Letting $a$ and $b$ represent concentrations of species $A$ and $B$, respectively, the mass-action mass-action ODEs~\eqref{eq:ODE-mass-action} are the following:
\begin{align*}
\frac{da}{dt} &= s_A - l_A a + (m-1)kab \\
\frac{db}{dt} &= s_B - l_B b + (n-1)kab
\end{align*}
Now let $N$ be an embedded network of $H_{m,n}$.  We first consider the case when at least one species (and possibly some reactions as well) are removed to obtain $N$.

\noindent (Species-removal case) If both species $A$ and $B$ are removed to obtain $N$, then $N$ is trivial.  So, assume without loss of generality that species $B$ is removed. Thus, $N$ is a subnetwork of $\{0 \leftrightarrow  A  \to mA\}$, so its steady state equation is a linear equation in $a$, and thus cannot have multiple solutions. So, we may assume that no species are removed from $H_{m,n}$ to obtain $N$, i.e.\ $N$ is a subnetwork of $H_{m,n}$.

\noindent (Reaction-removal case)  
We must show that if $N$ is a proper subnetwork of $H_{m,n}$, i.e.\ if one or more rate constants is set to zero, then $N$ is not multistationary. If one of the outflow rates $l_A$ or $l_B$ is zero, then there are no steady states as either $\dot a >0$ or $\dot b >0$ for all $(a,b)$ in the positive orthant. On the other hand, if $k=0$ or one of the inflow rates $s_A$ or $s_B$ is zero, then the steady state equations system can be solved explicitly, and there is at most one positive solution. 
%
\end{proof}

\begin{remark}
For the values of $m$ and $n$ for which $G_{m,n}$ or $H_{m,n}$ is multistationary, the parameter space (for the rate constants) for which there exist multiple positive steady states is identified in \cite[Lemmas 4.3 and 4.5]{joshi2013complete}.  
\end{remark}

The networks  $G_{m,n}$ and $H_{m,n}$ have only one non-flow irreversible reaction, and only one species (in the case of $G_{m,n}$) or two species ($H_{m,n}$). Both form infinite families, so the molecularities of the chemical species are unbounded. It may be argued that this is an unnatural property, and molecularity in realistic models must be small. In particular, if we restrict complexes to be at most bimolecular, then by Corollary~\ref{cor:cfstr-1-rxn} there are no CFSTR atoms of multistationarity that contain only one non-flow reaction (reversible or irreversible). 

In Section~\ref{subsec:seq-infin}, we will show that even with the bimolecular restriction there are infinitely many CFSTR atoms of multistationarity. Since molecularity is bounded, the numbers of reactions and species will be unbounded in this family. The demonstration is by explicit construction -- using certain ``sequestration networks'' introduced next. 

\subsection{Sequestration networks} \label{subsec:seq}
 This subsection is a slight detour, in which we introduce sequestration networks, and apply many of our earlier results to study them.
We will call reactions of the type $S+E \to 0$ {\em sequestration reactions}. Here we view $S$ as a substrate that binds with an enzyme $E$ and is then sequestered or rendered non-reactive. The reaction might be more accurately modeled as $S+E \to SE$, where $SE$ represents the substrate-enzyme complex. However, we assume that the complex dissociates very rarely, and does not participate in any other reaction within the network, either because the complex is inert or because the complex leaves the reaction vessel through some unspecified mechanism. In such an event, it is reasonable to model the reaction, for simplicity and without changing the dynamics, as $S+E \to 0$. 

Let us describe an instance of a sequestration reaction in {\em Escherichia coli}.  The trp operon is a sequence of five genes that codes for the amino acid tryptophan. The regulatory protein, called trp repressor, can bind in the presence of tryptophan to the operator site of the trp operon and prevent its transcription. Thus, the presence of tryptophan inhibits its own production in E.~coli \cite{santillan2001dynamic} -- we can model the process as a sequestration reaction with $E$ representing tryptophan and $S$ representing trp operon. 

An even more compelling example is found in neuron signaling mechanisms. Communication between a pair of neurons occurs via the medium of neurotransmitters released by the presynaptic neuron into the synaptic cleft. The signal is terminated by the action of  neurotransporters, which are neurotransmitter reuptake proteins that bind to neurotransmitter molecules and mediate their removal from the synaptic cleft \cite{lesch1995neurotransmitter}. In this case, $E$ represents neurotransporters and $S$ represents neurotransmitters. In general, even though an individual sequestration reaction is symmetric in $S$ and $E$, their distinct chemical roles are revealed when considered within a network. 

Next we define a
sequestration network, in which a {\em synthesis reaction} of the type $X_1 \to m X_n$ is coupled with $n-1$ sequestration reactions. When $n$ is even, the sequestration reactions generate negative feedback to the synthesis reaction, while when $n$ is odd, the feedback to the synthesis reaction is positive. We will establish that this positive feedback results in the capacity for multiple steady states (Theorem~\ref{thm:seques-net-mss}).
\begin{definition} \label{def:seq}
For positive integers $n \ge 2, m \ge 1$, we define the {\em sequestration network} $K_{m,n}$ {\em of order $n$ with production factor $m$} to be:
\begin{align}
X_1 &\to mX_n \nonumber \tag{6.2.1} \label{eq:Kn1}\\
X_1 + X_2 &\to 0  \nonumber \tag{6.2.2} \label{eq:Kn2}\\
\vdots \nonumber \\
X_{n-1} + X_n &\to 0 \nonumber \tag{6.2.n} \label{eq:Knn}
\end{align}
\end{definition}
The sequestration network $K_{m,n}$ has $n$ species, $n$ (non-flow) reactions, and for $1 \le m \le 2$, its complexes are at most bimolecular.  Also, 
$K_{2,n}$ is an embedded network of, and is inspired by, a reaction network analyzed by Schlosser and Feinberg~\cite[Table 1]{SchlosserFeinberg}.  

The following result classifies the multistationarity of the fully open extensions of the $K_{m,n}$'s; it will follow immediately from Lemmas~\ref{lem:m=1} and~\ref{lem:Kn-for-n-odd} and Theorem~\ref{thm:even-n-no-MSS}.

\begin{theorem} \label{thm:seques-net-mss} 
For positive integers $n \ge 2$ and $m \ge 1$, the fully open extension $\widetilde K_{m,n}$ of the sequestration network $ K_{m,n}$ is multistationary if and only if $m > 1$ and $n$ is odd.  
\end{theorem}

Our analysis of the sequestration networks begins with the following results.
\begin{lemma} \label{lem:m=1}
For any positive integer $n \ge 2$, if $m=1$, then the fully open extension $\widetilde K_{m,n}$ of the sequestration network $K_{m,n}$ is not multistationary. 
\end{lemma}
\begin{proof}
For all $m \ge 1$, the maximum total molecularity of any species in $K_{m,n}$ is $m+1$. Thus, if $m=1$, then the maximum total molecularity is 2, so by definition $K_{m,n}$ has no relevant square embedded networks. Therefore, by Theorem \ref{thm:inj-cfstr} the network does not admit multiple steady states. 
\end{proof}

\begin{lemma} \label{lemma:rel_sen}
For positive integers $n \ge 2$ and $m \ge 1$, the sequestration network $K_{m,n}$ has no proper, relevant square embedded network. 
\end{lemma}
\begin{proof}
Let $N$ be a relevant SEN of $K_{m,n}$. Then $N$ must contain the species $X_n$,  because otherwise the maximum total molecularity of $N$ is less than three and $N$ must contain the reactions \eqref{eq:Kn1} and \eqref{eq:Knn} because $X_n$ must appear in at least two complexes. $N$ must contain $X_1$ because otherwise \eqref{eq:Kn1} reduces to a generalized inflow reaction, which cannot occur in a relevant SEN.  This further implies that $N$ must contain \eqref{eq:Kn2} because $X_1$ must occur in at least two complexes, which further implies that $X_2$ must be in $N$, because otherwise \eqref{eq:Kn2} reduces to an outflow reaction. Continuing this process, we find that the species $X_1, \ldots, X_n$ must be contained in $N$ and the reactions \eqref{eq:Kn1} to \eqref{eq:Knn} must be in $N$. Thus $N=K_{m,n}$ is not a proper SEN of $K_{m,n}$.
%
%
\end{proof}
\begin{lemma} \label{lemma:negor}
For positive integers $n \ge 2$ and $m \ge 1$, the sequestration network $K_{m,n}$ is negatively oriented if and only if $m \ge 2$ and $n$ is odd. 
\end{lemma}
\begin{proof}
Write $K_{m,n}$ informally as $A_n \to B_n$, where $A_n$ is the source matrix and $B_n$ is the product matrix.   Thus, the orientation ${\rm Or}(K_{m,n})$ is the sign of the following:
\[
 \det A_n \det (A_n - B_n) :=
\det 
\begin{pmatrix} 
1 & 0 & 0 & \ldots & 0 & 0 \\
1 & 1 & 0 & \ldots & 0 & 0 \\
0 & 1 & 1 & \ldots & 0 & 0 \\
\vdots & \vdots & \vdots & \ldots & \vdots &  \vdots \\
0 & 0 & 0 & \ldots & 1 & 1
\end{pmatrix}
\det 
\begin{pmatrix} 
1 & 0 & 0 & \ldots & 0 & -m \\
1 & 1 & 0 & \ldots & 0 & 0 \\
0 & 1 & 1 & \ldots & 0 & 0 \\
\vdots & \vdots & \vdots & \ldots & \vdots &  \vdots \\
0 & 0 & 0 & \ldots & 1 & 1
\end{pmatrix}
= 1 \cdot \left( 1 + (-1)^{n+1} (-m) \right)~,
\]
which is negative if and only if $n$ is odd and $m \ge 2$.
\end{proof}

\begin{theorem} \label{thm:even-n-no-MSS}
For positive integers $n \ge 2$ and $m \ge 2$, if $n$ is even, then the fully open extension $\widetilde K_{m,n}$ of the sequestration network $K_{m,n}$ is not multistationary. Furthermore, no sub-CFSTR (a subnetwork that is a CFSTR) of $\widetilde K_{m,n}$ is multistationary. 
\end{theorem}
\begin{proof} By Lemma \ref{lemma:rel_sen}, $K_{m,n}$ has no proper, relevant SENs.  Thus, the only relevant SEN, if $n$ is even, is $K_{m,n}$ itself, which is positively oriented by Lemma \ref{lemma:negor}. Thus, by Theorem \ref{thm:inj-cfstr}, $\widetilde K_{m,n}$ is not multistationary, nor is any subnetwork of $\widetilde K_{m,n}$ that is also a CFSTR (because relevant SENs of such a network are relevant SENs of $K_{m,n}$).
\end{proof}

In the next section, we show that $\widetilde K_{m,n}$ is multistationary when $n$ is an odd. 

\subsection{New result: an infinite family with complexes that are at most bimolecular} \label{subsec:seq-infin}
Here we show that the fully open extension of the sequestration network $K_{m,n}$, which has $n$ reactions and $n$ species, admits multiple positive steady states when $n$ is odd.
\begin{lemma} \label{lem:Kn-for-n-odd}
For positive integers $n \ge 2$ and $m \ge 2$, if $n$ is odd, then 
\begin{enumerate}
	\item $\widetilde K_{m,n}$ is a multistationary CFSTR, and 
	\item no fully open network that is an embedded network of  $\widetilde K_{m,n}$ (besides  $\widetilde K_{m,n}$ itself) is multistationary.
\end{enumerate}
\end{lemma}

\begin{proof}
By 
Lemma~\ref{lemma:negor} together with Theorem~\ref{thm:det-opt}, to show that $\widetilde K_{m,n}$ is multistationary (for $m \geq 2$ and $n$ odd), we need only exhibit a positive linear combination of the reaction vectors $(1, 0, \ldots, 0, -m), (1,1,0, \ldots, 0,0), (0,1,1, 0 \ldots, 0,0), \ldots, (0,0,\ldots, 0,1,1)$ which is a positive vector. This can be achieved by choosing as coefficients $(\eta_1, \ldots, \eta_n) = (1,1,\ldots,1,m+1)$. 

For part 2, let $N$ be a fully open network that is an embedded network of $\widetilde K_{m,n}$.  Every SEN of the non-flow subnetwork $N^{\circ}$ of $N$ is an SEN of $K_{m,n}$, so in light of Theorem~\ref{thm:inj-cfstr} and Lemmas~\ref{lemma:rel_sen} and~\ref{lemma:negor}, $N^{\circ}$ must equal $K_{m,n}$ in order for $N$ to be multistationary, i.e.\ $N=\widetilde K_{m,n}$.  
\end{proof}

The following conjecture is what remains for us to show that the networks $\widetilde K_{m,n}$ in Lemma~\ref{lem:Kn-for-n-odd} are CFSTR atoms of multistationarity.  It is true for $n=3$ and $m=2$.

\begin{conjecture} \label{conj}
For positive integers $n \ge 2$ and $m \ge 2$, if $n$ is odd, then $\widetilde K_{m,n}$ admits multiple {\em nondegenerate} steady states. 
 \end{conjecture}

\begin{theorem} \label{thm:Kn-cfstr-atom}
If Conjecture~\ref{conj} holds, then:
\begin{enumerate}
	\item For positive integers $n \ge 2$ and $m \ge 2$, if $n$ is odd, then $\widetilde K_{m,n}$ is a CFSTR atom of multistationarity.
	\item There exist infinitely many embedding-minimal multistationary networks with complexes that are at most bimolecular.
\end{enumerate}
\end{theorem} 

\begin{proof}
Assume that Conjecture~\ref{conj} holds.  Then part 1 follows from part 2 of Lemma~\ref{lem:Kn-for-n-odd}.  

It follows that the $\widetilde K_{2,n}$'s, for odd $n\ge 2$, form an infinite family of at-most-bimolecular CFSTR atoms of multistationarity. 
By definition, each such $\widetilde K_{2,n}$ has at least one subnetwork $N=N_{2,n}$ which is an embedding-minimal multistationary network. 

We now claim that $K_{2,n}$ is a subnetwork of $N$. By way of contradiction, assume that this is not the case, so $N^{\circ}$, the non-flow subnetwork of $N$, is a proper, embedded network of $K_{2,n}$. Consider a square embedded network $N_i$ of $N$. If $N_i$ contains an inflow reaction, then its orientation is zero. On the other hand, if $N_i$ contains an outflow reaction, the species in the outflow reaction can be removed without changing the orientation of $N_i$. Thus, to establish injectivity of $N$, it suffices to show that all SENs of $N^\circ$ have nonnegative orientation (recall the discussion after Definition~\ref{def:square-or}), which is furthermore equivalent to showing that all relevant SENs of $N^\circ$ have nonnegative orientation. Indeed, this follows from Lemma~\ref{lemma:rel_sen},  as no proper SEN of $K_{2,n}$ is relevant. 


Thus, $N_{2,n}$ has $n$ species, so these embedding-minimal multistationary networks (for odd $n$) form an infinite family.
\end{proof}


 
\section{Existence of infinitely many embedding-minimal multistable networks} \label{sec:infinitely2}

In this section, we demonstrate the existence of infinitely many embedding-minimal {\em multistable} networks, again via explicit construction. We will build on the findings for the networks $G_{m,n}$ 
studied earlier. 


\begin{theorem} \label{thm:multistable}
For positive integers $m \ne n$, consider the network: $\overline{G}_{m,n} := \{0 \leftrightarrow  A,~ mA \lra nA \}$. 
Then $\overline G_{m,n}$ is a CFSTR atom of multistability and an embedding-minimal multistable network if and only if $ n > 1$ and $m > 1$. 
\end{theorem} 
\begin{proof}
We assume without loss of generality that $n > m$. We know from Theorem \ref{thm:onerxn} that $\overline G_{m,n}$ is multistationary if and only if $n > m >1$. Let  $s$, $l$, $k_+$ and $k_-$ denote the rates of the reactions $0 \to A$, $A \to 0$, $mA \to nA$, and $nA \to mA$, respectively. 
Then, the single mass-action ODE~\eqref{eq:ODE-mass-action} arising from $\overline G_{m,n}$ is $\dot a = s - la + (n-m)k_+a^m - (n-m)k_-a^n$, where $a$ denotes the concentration of the species $A$; thus the number of positive steady states is the number of positive real roots of the univariate polynomial 
\begin{align} \label{eq:poly-for-G2}
f(a) := s - la + (n-m)k_+a^m- (n-m)k_-a^n~.
\end{align}
Note that $\overline G_{m,n}$ reduces to $G_{m,n}$ when $k_- = 0$. By Theorem \ref{thm:onerxn}, there exist parameter values $s, l, k_+>0$ for which $G_{m,n}$ has two nondegenerate positive steady states. Thus, for fixed  $s, l, k_+>0$ such that $G_{m,n}$ has two nondegenerate positive steady states, there exists a sufficiently small $k_->0$ such that $\overline G_{m,n}$ has three nondegenerate positive steady states. Furthermore, since $f(0) = s >0$ and $\lim_{a \to \infty} f(a) = -\infty$, the graph of $f(a)$ crosses the $a$-axis from above twice -- these crossing points correspond to stable steady states. 

The only nontrivial embedded networks of $\overline G_{m,n}$ arise from removing reactions (as $G_{m,n}$ has only one species), i.e. setting one of the reaction rates $s$, $l$, $k_+$ or $k_-$ to zero.  It is straightforward to see that this results in the polynomial~\eqref{eq:poly-for-G2} having at most two positive real roots, of which at most one corresponds to a stable steady state. 
\end{proof}

\section{Discussion} \label{sec:openQ}
As described in the Introduction, deciding whether a given network arising in practice is multistationary is an important first step in understanding its dynamics.  Here we have presented various criteria for answering this question.  However, there are many networks for which none of the existing criteria apply.  To this end, we have seen that the approach via ``transferring'' multistationarity from one network to a (typically larger) network can be helpful.  Nonetheless, we are still far from a complete answer.

Looking forward, we highlight some future directions:
\been
	\item Future investigations will likely develop additional criteria for deciding multistationarity tailored to networks arising in specific application domains.  
	\item Some networks arising in biology are CFSTRs.  Can we use the idea of CFSTR atoms or embedding-minimal multistationary networks to understand such networks that are multistationary?  As an example, Fouchet and Regoes analyzed a multistationary CFSTR arising in immunology~\cite{FouchetRegoes}; can we enumerate, and then interpret, the multistationary networks embedded in this network?  Siegel-Gaskins {\em et al.}\ perfomed a related analysis of small gene regulatory networks~\cite{Emergence}; see also~\cite[Remark 3.3]{joshi2012atoms}.
	\item We must address computational challenges inherent in enumerating CFSTR atoms and embedding-minimal multistationary networks, and then providing a method for determining whether a given network contains as an embedded network, a CFSTR atom, for instance.
\enen

Finally, our true interest is in atoms of {\em multistability}, as steady states observed in practice are necessarily stable.  Therefore, we need more criteria, guided by networks arising in practice, for when such atoms can be lifted.


\section*{Appendix: properties of the embedded-network relation}
Here we prove Proposition~\ref{prop:emb-properties}, restated here as Proposition~\ref{prop:restated}.

\begin{lemma} \label{lem:emb-properties}
Let $N$ be an embedded network of $G$. 
\been
\item If $N$ is a subnetwork of $G$, then the stoichiometric subspace of $N$ is a subspace of the stoichiometric subspace of $G$. 
\item If $N$ is obtained by removing a set of species of $G$, then the stoichiometric subspace of $N$ is a projection of the stoichiometric subspace of $G$. 
\item The number of complexes in $N$ is no more than the number of complexes in $G$. 
\enen
\end{lemma}
\begin{proof}
The stoichiometric subspace of a network $N$ is $S_N := \mbox{span} \{y' - y  |  y \to y' \in N\}$. If $N$ is a subnetwork of $G$, $\{y' - y | y \to y' \in N\} \subseteq \{y' - y | y \to y' \in G\}$, and so $S_N$ is a subspace of $S_G$. If $N$ is obtained by removing a subset of species of $G$, the reaction vectors of $N$ are a projection of the reaction vectors of $G$, from which the second assertion follows. 
As for the third item, if $N$ is a subnetwork of $G$, then the complexes of $N$ are a subset of the complexes of $G$. In the case of species removal, every complex in $N$ is a projection of a complex in $G$, so this completes the proof. 
\end{proof}

\begin{proposition} \label{prop:restated}
\been
\item If $N$ is an embedded network of $G$, the dimension of the stoichiometric subspace of $N$ is less than or equal to the dimension of the stochiometric subspace of $G$, $\dim(S_N) \le \dim(S_G)$. 
\item If $N$ is a subnetwork of $G$, then the deficiency of $N$ is no greater than the deficiency of $G$, that is, $\delta_N \le \delta_G$. 
\enen
\end{proposition}
\begin{proof}
The first assertion is a direction consequence Lemma~\ref{lem:emb-properties}. For the second item, we need only show that if $N$ is obtained from $G$ by removing one reaction, then $\delta_N \le \delta_G$. In fact, we will show in this case that $\delta_N \in \{\delta_G, \delta_{G}-1 \}$. Suppose that $N = G \setminus \{ y \to y'\}$, i.e. $\{ y \to y'\}$ is the reaction that is removed from $G$ to obtain $N$. If $y' - y$ is in the span of the reaction vectors of $N$, then $\dim(S_N) = \dim(S_G)$, otherwise $\dim(S_N) = \dim(S_G) - 1$. Let $p_N$ and $p_G$ represent the number of complexes, and let $l_N$ and $l_G$ represent the number of linkage classes of $N$ and $G$, respectively. 

Case (i) ($p_N = p_G-2$): Removing the single reaction $\{ y \to y'\}$ from $G$ results in removing two complexes only if neither $y$ nor $y'$ participates in any reaction of $N$. In other words, this case occurs only if $\{ y \to y'\}$ is itself a linkage class. Therefore, $l_N = l_G-1$, and the result follows. 

Case (ii) ($p_N = p_G-1$): In this case, either $y$ or $y'$, but not both of them, participates in at least one reaction of $N$, which implies that $l_N = l_G$. 

Case (iii) ($p_N = p_G$): In this case, each of $y$ and $y'$ participates in at least one reaction of $N$. Suppose that removal of $y \to y'$ from $G$ results in a linkage class of $G$ splitting into two linkage classes of $N$, so that $l_N = l_G+1$. Then $\delta_N \in \{\delta_G, \delta_G-1\}$. If removal of $y \to y'$ from $G$ does not split a linkage class, then $l_N = l_G$. Since the linkage class does not split, there exists a set of complexes in $N$, $\{y_0,y_1, \ldots y_{m-1}, y_m\}$ where $y_0 :=y$ and $y_m :=y'$ such that either $y_i \to y_{i+1}$ or $y_{i+1} \to y_i$ is a reaction in $N$ for all $i$ for which $0 \le i \le m-1$. But this implies that the reaction vector $y' - y$ is a linear combination of the reaction vectors $\pm (y_i - y_{i+1})$. Thus, $\dim(S_N) = \dim(S_G)$, and so $\delta_N = \delta_G$. 
\end{proof}

\begin{remark}
Removing a species from a reaction network, even if doing so does not  decrease the number of reactions, can cause the deficiency to increase, decrease, or remain the same.  For instance, removing species $A$ from $\{ A+B \to 0, ~ A \to 2B \}$ yields the embedded network $\{ B \to 0 \to 2B\}$ and increases the deficiency from 0 to 1. 
An example in which the deficiency does not change arises by removing $A$ from the reaction $A \to B$; both have deficiency 0.  
Finally, removing the species $A$ from the network $\{A+B \to C \to B \to D \to 2A+B\}$ yields the network $\{C \lra B \lra D \}$, which decreases the deficiency from 1 to 0.
\end{remark} 


\subsection*{Acknowledgements}
AS was supported by the NSF (DMS-1312473).



\bibliographystyle{plain}
\bibliography{multistationarity}

\begin{thebibliography}{10}

\bibitem{ADS10}
David Angeli, Patrick De~Leenheer, and Eduardo Sontag.
\newblock Graph-theoretic characterizations of monotonicity of chemical
  networks in reaction coordinates.
\newblock {\em J. Math. Biol.}, 61(4):581--616, 2010.

\bibitem{banaji}
Murad Banaji.
\newblock Monotonicity in chemical reaction systems.
\newblock {\em Dyn. Syst.}, 24(1):1--30, 2009.

\bibitem{BanajiCraciun2010}
Murad Banaji and Gheorghe Craciun.
\newblock {Graph-theoretic criteria for injectivity and unique equilibria in
  general chemical reaction systems}.
\newblock {\em Adv. Appl. Math.}, 44(2):168--184, 2010.

\bibitem{BanajiM}
Murad Banaji and Janusz Mierczy{\'n}ski.
\newblock Global convergence in systems of differential equations arising from
  chemical reaction networks.
\newblock {\em J. Differential Equations}, 254(3):1359--1374, 2013.

\bibitem{BP}
Murad Banaji and Casian Pantea.
\newblock Some results on injectivity and multistationarity in chemical
  reaction networks.
\newblock {\em preprint, {\tt http://arXiv.org/abs/1309.6771}}, 2013.

\bibitem{ConradiMincheva14}
C.~Conradi and M.~Mincheva.
\newblock Graph-theoretic analysis of multistationarity using degree theory.
\newblock {\em preprint, {\tt http://arxiv.org/abs/1411.2896}}, 2014.

\bibitem{Conradi_subnetwork}
Carsten Conradi, Dietrich Flockerzi, J\"org Raisch, and J\"org Stelling.
\newblock {Subnetwork analysis reveals dynamic features of complex
  (bio)chemical networks}.
\newblock {\em Proc. Natl. Acad. Sci. USA}, 104(49):19175--19180, 2007.

\bibitem{CGS}
G.~Craciun, L.~Garcia-Puente, and F.~Sottile.
\newblock Some geometrical aspects of control points for toric patches.
\newblock In M~D\ae{}hlen, M~S Floater, T~Lyche, J-L Merrien, K~Morken, and L~L
  Schumaker, editors, {\em Mathematical Methods for Curves and Surfaces},
  volume 5862 of {\em Lecture Notes in Comput. Sci.}, pages 111--135,
  Heidelberg, 2010. Springer.

\bibitem{ME1}
Gheorghe Craciun and Martin Feinberg.
\newblock Multiple equilibria in complex chemical reaction networks. {I}. {T}he
  injectivity property.
\newblock {\em SIAM J. Appl. Math.}, 65(5):1526--1546, 2005.

\bibitem{ME_entrapped}
Gheorghe Craciun and Martin Feinberg.
\newblock Multiple equilibria in complex chemical reaction networks: extensions
  to entrapped species models.
\newblock {\em IEE Proceedings-Systems Biology}, 153:179--186, 2006.

\bibitem{ME2}
Gheorghe Craciun and Martin Feinberg.
\newblock Multiple equilibria in complex chemical reaction networks. {II}.
  {T}he species-reaction graph.
\newblock {\em SIAM J. Appl. Math.}, 66(4):1321--1338, 2006.

\bibitem{ME3}
Gheorghe Craciun and Martin Feinberg.
\newblock Multiple equilibria in complex chemical reaction networks: Semiopen
  mass action systems.
\newblock {\em SIAM J. Appl. Math.}, 70(6):1859--1877, 2010.

\bibitem{CPS}
Gheorghe Craciun, Casian Pantea, and Eduardo Sontag.
\newblock {\em {Graph-theoretic characterizations of multistability and
  monotonicity for biochemical reaction networks}}, pages 63--72.
\newblock Springer, 2011.

\bibitem{DB}
Pete Donnell and Murad Banaji.
\newblock Local and global stability of equilibria for a class of chemical
  reaction networks.
\newblock {\em SIAM J. Appl. Dyn. Syst.}, 12(2):899--920, 2013.

\bibitem{control}
Pete Donnell, Murad Banaji, Anca Marginean, and Casian Pantea.
\newblock {CoNtRol}: an open source framework for the analysis of chemical
  reaction networks.
\newblock {\em Bioinformatics}, page to appear, 2014.

\bibitem{EllisonThesis}
Phillipp Ellison.
\newblock {\em The advanced deficiency algorithm and its applications to
  mechanism discrimination}.
\newblock PhD thesis, University of Rochester, 1998.

\bibitem{Toolbox}
Phillipp Ellison, Martin Feinberg, Haixia Ji, and Daniel Knight.
\newblock {C}hemical {R}eaction {N}etwork {T}oolbox, 2011.
\newblock Available at \url{http://www.crnt.osu.edu/CRNTWin}.

\bibitem{Feinberg72}
Martin Feinberg.
\newblock Complex balancing in general kinetic systems.
\newblock {\em Arch. Rational Mech. Anal.}, 49(3):187--194, 1972.

\bibitem{FeinDefZeroOne}
Martin Feinberg.
\newblock {Chemical reaction network structure and the stability of complex
  isothermal reactors I. The deficiency zero and deficiency one theorems}.
\newblock {\em Chem. Eng. Sci.}, 42(10):2229--2268, 1987.

\bibitem{Fein95DefOne}
Martin Feinberg.
\newblock Multiple steady states for chemical reaction networks of deficiency
  one.
\newblock {\em Arch. Rational Mech. Anal.}, 132(4):371--406, 1995.

\bibitem{FeliuWiuf_MAK}
E.~Feliu and C.~Wiuf.
\newblock Preclusion of switch behavior in reaction networks with mass-action
  kinetics.
\newblock {\em Appl. Math. Comput.}, 219:1449--1467, 2012.

\bibitem{feliu}
Elisenda Feliu.
\newblock Injectivity, multiple zeros, and multistationarity in reaction
  networks.
\newblock {\em preprint, {\tt http://arXiv.org/abs/1407.2955}}, 2014.

\bibitem{FeliuWiuf}
Elisenda Feliu and Carsten Wiuf.
\newblock Simplifying biochemical models with intermediate species.
\newblock {\em J. R. Soc. Interface}, 10(87), 2013.

\bibitem{FouchetRegoes}
David Fouchet and Roland Regoes.
\newblock A population dynamics analysis of the interaction between adaptive
  regulatory {T} cells and antigen presenting cells.
\newblock {\em PLoS ONE}, 3(5):e2306, 05 2008.

\bibitem{gnacadja_linalg}
G.~Gnacadja.
\newblock A {J}acobian criterion for the simultaneous injectivity on positive
  variables of linearly parameterized polynomials maps.
\newblock {\em Linear Algebra Appl.}, 437:612--622, 2012.

\bibitem{grassberger1982phase}
Peter Grassberger.
\newblock On phase transitions in {S}chl{\"o}gl's second model.
\newblock {\em Zeitschrift f{\"u}r Physik B Condensed Matter}, 47(4):365--374,
  1982.

\bibitem{HeltonDeterminant}
J.~William Helton, Igor Klep, and Raul Gomez.
\newblock {Determinant expansions of signed matrices and of certain Jacobians}.
\newblock {\em SIAM J. Matrix Anal. Appl.}, 31(2):732--754, 2009.

\bibitem{ivanova}
A.~Ivanova.
\newblock Conditions for uniqueness of stationary state of kinetic systems
  related to structural scheme of reactions.
\newblock {\em Kinet.\ Katal.}, 20:1019--1023, 1979.

\bibitem{JiThesis}
Haixia Ji.
\newblock {\em Uniqueness of Equilibria for Complex Chemical Reaction
  Networks}.
\newblock PhD thesis, Ohio State University, 2011.

\bibitem{joshi2013complete}
Badal Joshi.
\newblock Complete characterization by multistationarity of fully open networks
  with one non-flow reaction.
\newblock {\em Applied Mathematics and Computation}, 219:6931--6945, 2013.

\bibitem{simplifying}
Badal Joshi and Anne Shiu.
\newblock Simplifying the {J}acobian criterion for precluding multistationarity
  in chemical reaction networks.
\newblock {\em SIAM J. Appl. Math.}, 72(3):857--876, 2012.

\bibitem{joshi2012atoms}
Badal Joshi and Anne Shiu.
\newblock Atoms of multistationarity in chemical reaction networks.
\newblock {\em Journal of Mathematical Chemistry}, 51(1):153--178, 2013.

\bibitem{lesch1995neurotransmitter}
K~Peter Lesch and Dietmar Bengel.
\newblock Neurotransmitter reuptake mechanisms.
\newblock {\em CNS drugs}, 4(4):302--322, 1995.

\bibitem{MY}
Guy Marin and Gregory~S. Yablonsky.
\newblock {\em Kinetics of Chemical Reactions}.
\newblock Wiley-VCH, Wienheim, Germany, 2011.

\bibitem{MinchevaCraciun2008}
M.~Mincheva and G.~Craciun.
\newblock Multigraph conditions for multistability, oscillations and pattern
  formation in biochemical reaction networks.
\newblock {\em Proceedings of the IEEE}, 96(8):1281--1291, aug. 2008.

\bibitem{signs}
Stefan M\"uller, Elisenda Feliu, Georg Regensburger, Carsten Conradi, Anne
  Shiu, and Alicia Dickenstein.
\newblock Sign conditions for injectivity of generalized polynomial maps with
  applications to chemical reaction networks and real algebraic geometry.
\newblock {\em to appear in Found. Comput. Math.}

\bibitem{santillan2001dynamic}
Mois{\'e}s Santill{\'a}n and Michael~C Mackey.
\newblock Dynamic regulation of the tryptophan operon: a modeling study and
  comparison with experimental data.
\newblock {\em Proceedings of the National Academy of Sciences},
  98(4):1364--1369, 2001.

\bibitem{schlogl1972chemical}
Friedrich Schl{\"o}gl.
\newblock Chemical reaction models for non-equilibrium phase transitions.
\newblock {\em Zeitschrift f{\"u}r Physik}, 253(2):147--161, 1972.

\bibitem{SchlosserFeinberg}
Paul~M. Schlosser and Martin Feinberg.
\newblock A theory of multiple steady states in isothermal homogeneous {CFSTR}s
  with many reactions.
\newblock {\em Chemical Engineering Science}, 49(11):1749--1767, 1994.

\bibitem{ShinarFeinberg2012}
G.~Shinar and M.~Feinberg.
\newblock Concordant chemical reaction networks.
\newblock {\em Math. Biosci.}, 240(2):92--113, 2012.

\bibitem{ShinarFeinberg2013}
G.~Shinar and M.~Feinberg.
\newblock Concordant chemical reaction networks and the species-reaction graph.
\newblock {\em Math. Biosci.}, 241(1):1--23, 2013.

\bibitem{Emergence}
Dan Siegal-Gaskins, Maria~Katherine Mejia-Guerra, Gregory~D. Smith, and Erich
  Grotewold.
\newblock Emergence of switch-like behavior in a large family of simple
  biochemical networks.
\newblock {\em PLoS Comput. Biol.}, 7(5):e1002039, 05 2011.

\bibitem{WiufFeliu_powerlaw}
C.~Wiuf and E.~Feliu.
\newblock Power-law kinetics and determinant criteria for the preclusion of
  multistationarity in networks of interacting species.
\newblock {\em SIAM J. Appl. Dyn. Syst.}, 12:1685--1721, 2013.

\bibitem{kinetic-book}
G.S. Yablonskii, V.I. Bykov, A.N. Gorban, and V.I. Elokhin.
\newblock {\em Kinetic Models of Catalytic Reactions}.
\newblock Comprehensive Chemical Kinetics, vol.\ 32, ed. by R.G. Compton,
  Elsevier, Amsterdam, 1991.

\end{thebibliography}

%
%
%
%
%
%

\end{document}